\documentclass[11pt,letterpaper]{article}

 \usepackage[hscale=0.75,vscale=0.8]{geometry}

\usepackage[utf8]{inputenc} 
\usepackage[T1]{fontenc}    
\usepackage{hyperref}       
\usepackage{url}            
\usepackage{booktabs}       
\usepackage{amsfonts}       
\usepackage{nicefrac}       

\usepackage{microtype}
\usepackage{graphicx}
\usepackage{wrapfig}
\usepackage{subcaption}
\usepackage{natbib}

\usepackage{amssymb,amsmath,color,bm,algorithm,algorithmic}

\usepackage[linesnumbered,ruled,vlined,algo2e]{algorithm2e}

\SetCommentSty{mycommfont}

\SetKwInput{KwInput}{Input}                
\SetKwInput{KwOutput}{Output}              

\usepackage{url}

\usepackage[mathcal]{eucal}

\DeclareMathAlphabet\mathbfcal{OMS}{cmsy}{b}{n}

\newcommand{\BEAS}{\begin{eqnarray*}}
\newcommand{\EEAS}{\end{eqnarray*}}
\newcommand{\BEA}{\begin{eqnarray}}
\newcommand{\EEA}{\end{eqnarray}}
\newcommand{\BEQ}{\begin{equation}}
\newcommand{\EEQ}{\end{equation}}
\newcommand{\BIT}{\begin{itemize}}
\newcommand{\EIT}{\end{itemize}}
\newcommand{\BNUM}{\begin{enumerate}}
\newcommand{\ENUM}{\end{enumerate}}
\newcommand{\BA}{\begin{array}}
\newcommand{\EA}{\end{array}}

\newcommand{\argmin}{\mathop{\rm argmin}}

\newcommand{\idm}{I}
\newcommand{\rb}{\mathbb{R}}

\newcommand{\BlackBox}{\rule{1.5ex}{1.5ex}}  

\newenvironment{proof}{\par\noindent{\bf Proof\ }}{\hfill\BlackBox\\[2mm]}
\newtheorem{lemma}{Lemma}
\newtheorem{theorem}{Theorem}
\newtheorem{proposition}{Proposition}
\newtheorem{corollary}{Corollary}

\newcommand{\eq}[1]{Eq.~(\ref{eq:#1})}

\def \E{{\mathbb E}}

\def \Y{{\mathcal Y}}

\allowdisplaybreaks[4]
\newcommand{\punt}[1]{}

\newtheorem{cor}{Corollary}[section]

\def\argmin{\mathop{\rm arg\,min}}

\newcommand{\reals}{\mathbb{R}}

\def\argmin{\mathop{\rm arg\,min}}

\newcommand{\vx}{\mathbf{x}}

\newcommand{\vtheta}{\bm{\theta}}

\newcommand{\vz}{\mathbf{z}}

\newcommand{\bq}{\begin{equation}}
\newcommand{\ba}{\begin{eqnarray}}
\newcommand{\ea}{\end{eqnarray}}

\def\R{{\reals}}

\newcommand{\remove}[1]{}

\usepackage{cleveref}

\title{Explicit Regularization of Stochastic Gradient Methods\\ through Duality}

\author{Anant Raj \thanks{The work was done when Anant Raj was visiting Inria.} \\
  MPI for Intelligent Systems,\\
  T\"ubingen, Germany.\\
  \texttt{anant.raj@tuebingen.mpg.de} 
\and Francis Bach \\
  Inria, Ecole Normale Sup\'erieure \\
  PSL Research University, Paris, France. \\
  \texttt{francis.bach@inria.fr} \\
}

\begin{document}

\maketitle

\begin{abstract}
 We consider stochastic gradient methods under the interpolation regime where a perfect fit can be obtained (minimum loss at each observation). While previous work  highlighted the implicit regularization of such algorithms, we consider an explicit regularization framework as a minimum Bregman divergence convex feasibility problem. Using convex duality, we propose randomized Dykstra-style algorithms based on randomized dual coordinate ascent. For non-accelerated coordinate descent, we obtain an algorithm which bears strong similarities with (non-averaged) stochastic mirror descent on specific functions, as it is is equivalent for quadratic objectives, and equivalent in the early iterations for more general objectives. It comes with the benefit of an explicit convergence theorem to a minimum norm solution.  For accelerated coordinate descent, we obtain a new algorithm that has better convergence properties than existing stochastic gradient methods in the interpolating regime. This leads to accelerated versions of the perceptron for generic $\ell_p$-norm regularizers, which we illustrate in experiments. 
\end{abstract}

\section{Introduction}
With the recent advancement in machine learning and hardware research, the size and capacity of training models for machine learning tasks have been consistently increasing. For many model which is being widely used in practice, e.g.,  deep neural networks \citep{Goodfellow-et-al-2016} and non-parametric regression models~\citep{belkin2018does,liang2018just}, the training process achieves zero error, which means that such models are expressive enough to interpolate the training data completely. Hence, it is important  to understand the interpolation regime to improve the training and prediction of such complex and over parameterized models used in  machine learning. 

It is a well known fact that regularization, either explicit or implicit, plays a crucial role in achieving better generalization. While Tikhonov regularization is amongst the most famous form of regularization~\citep{golub1999tikhonov,weese1993regularization} for linear or non-linear problems, several other methods can induce regularization in form of computational regularization when training machine learning models~\citep{yao2007early,rudi2015less,srivastava2014dropout}. Apart from explicitly induced regularization in machine learning models, optimization algorithms like (stochastic) gradient descent which is widely used in practice while training large machine learning models, also induce implicit regularization in the obtained solution. In many cases, (stochastic) gradient descent converges to minimum Euclidean norm solutions. Recent series of papers~\citep{soudry2018implicit,gunasekar2018characterizing,kubo2019implicit,arora2019implicit}  present result about introducing implicit regularization/bias by (stochastic) gradient descent in different set of convex and non-convex problems. 

In this paper, we address the following question: instead of relying on implicit regularization properties of stochastic algorithms, can we introduce an \emph{explicit} regularization/bias while training over-parameterized models in the interpolation regime?

In optimization terms, the interpolation regime corresponds to the minimization of an average of finitely many functions of the form
$$ F(\theta) = \frac{1}{n} \sum_{i=1}^n f_i(\theta)$$ with respect to $\theta \in \R^d$, where there is a global minimizer of $F$, which happens to be a global minimizer of \emph{all} functions~$f_i$, for  $i \in \{1,\dots,n\}$ (instead of only minimizing their average). In the interpolation regime, we are thus looking for a point $\theta \in \R^d$ in the intersection of all sets of minimizers $$\mathcal{K}_i = \arg\min_{\eta \in \rb^d} f_i(\eta),$$ for  all $i \in\{1,\dots,n\}$. 

We can thus explicitly regularize the problem by solving the following optimization problem:
\begin{align}
\min_{ \theta \in \rb^d} \psi(\theta) \mbox{ such that } \forall i \in \{1,\dots,n\}, \ \theta \in \mathcal{K}_i, \label{eq:general_dual_formulation}
\end{align}
where $\psi$ is a regularization function (typically a squared norm). In the reformulated problem given in Eq.~\eqref{eq:general_dual_formulation}, explicit regularization can be induced in the solution via the structure of the function $\psi$.  Note also that the above problem can be seen as problem of generalized projection onto sets, which are convex if the original functions $f_i$'s are convex, which we assume throughout this paper.

To address the problem defined in Eq.~\eqref{eq:general_dual_formulation}, we use the tools from convex duality and accelerated randomized coordinate ascent, which results in Dykstra-style projection algorithms  \citep{boyle1986method,zhang2008successive,gaffke1989cyclic}.  In this paper, we make the following contributions:

\BIT
\vspace{-1mm}
\item[(a)] We provide  a generic inequality going from dual guarantees in function values to primal guarantees in terms of Bregman divergences of iterates.
\vspace{-1mm}
\item[(b)] For non-accelerated coordinate ascent, we obtain an algorithm which bears strong similarities with (non-averaged) stochastic mirror descent on specific functions $f_i$'s. Our algorithm comes with the benefit of an explicit convergence theorem to a minimum value of the regularizer.
\vspace{-1mm}
\item[(c)] For accelerated coordinate ascent, we obtain a new algorithm that has better convergence properties than existing stochastic gradient methods in the interpolating regime.
\vspace{-1mm}
\item[(d)] This leads to accelerated versions of the perceptron for generic $\ell_p$-norm regularizers (this is already an improvement for the $\ell_2$-regularizer).
\vspace{-1mm}
\EIT

\subsection{Related work}

\paragraph{Stochastic gradient methods.}
First order stochastic gradient based iterative approaches \citep{nemirovski2009robust,duchi2011adaptive,kingma2014adam,defazio2014saga,ward2019adagrad} are the most efficient methods to perform optimization for machine learning problems with large datasets. There has been a large amount of work done in the area of  stochastic first order optimization  methods \citep[see, e.g.,][and references therein]{polyak1990new,polyak1992acceleration,nemirovski2009robust,moulines2011non} since the original stochastic approximation approach was proposed  by \citet{robbins1951stochastic}.

\paragraph{Primal SGD in the interpolation regime.} To address the optimization problem in the interpolation regime,  \citet{vaswani2018fast} provide faster convergence rates for first order stochastic methods in the Euclidean geometry. They propose a  {strong growth condition},    and a more widely applicable  {weak growth condition},   under which stochastic gradient descent algorithm achieves fast convergence rate while using constant learning rate (a side contribution of our paper is to extend the latter algorithm to stochastic mirror descent). \citet{vaswani2019painless} propose to use line-search to  set the step-size while training over-parameterized models which can fit completely to data.  Several other works propose to use constant learning rate for stochastic gradient  methods \citep{ma2017power,bassily2018exponential,liu2018accelerating,cevher2019linear} while training extremely expressive models which interpolate. However, all of the above mentioned works are primal-based algorithms.  

\paragraph{Dysktra's projection algorithms.} Dykstra-type projection algorithms \citep{boyle1986method,gaffke1989cyclic} are simple modifications of the classical alternating projections methods~\citep{vonfunctional,halperin1962product} to project on the intersection of convex sets. A key interpretation is the connection between Dykstra’s algorithm and block coordinate ascent~\citep{bauschke2015projection,bauschke2011convex,tibshirani2017dykstra}, which we use in this paper. \citet{chambolle2017accelerated} provides accelerated rates for Dykstra projection algorithm when projecting on the intersection of two sets.

\paragraph{Coordinate descent.} Coordinate descent has a long history in the optimization literature \citep{tseng1987relaxation,tseng1993dual,tseng2001convergence}. Rates for accelerated randomized coordinate descent were first proved by \citet{nesterov2012efficiency}. Since then, various extensions of the accelerated coordinate descent including proximal accelerated coordinate descent and non-uniform sampling have been proposed by \citet{lin2015accelerated,allen2016even,nesterov2017efficiency,hendrikx2019accelerated}.  Dual coordinate ascent can also be used to solve regularized empirical risk minimization problem ~\citep{shalev2013stochastic,shalev2014accelerated}. We recover some of their results as a by-product in this paper.

\paragraph{Perceptron.} The perceptron is one of the oldest machine learning algorithms \citep{block1962perceptron,minsky2017perceptrons}. Since then, there has been a lot of work on theoretical and empirical foundations of perceptron algorithms \citep{freund1999large,shalev2005new,tsampouka2005analysis}, in particular, with related extensions to ours, to $\ell_p$-norm perceptron through mirror maps~\citep{grove2001general,kivinen2003online}. However, none of the above mentioned work forces structure to the optimal solution in an explicit way.

\section{Optimization Algorithms for Finite Data}
We consider the finite data setting, that is, we will give bounds on training objectives (or distances to the minimum norm interpolator on the training set). We thus consider the problem:
\begin{align}
\min_{\theta \in \rb^d} \Psi(\theta) \mbox{ such that }  \forall i \in \{1,\dots,n\}, \ x_i^\top \theta \in \mathcal{Y}_i, \label{eq:just_primal}
\end{align}
where:
\vspace{-1mm}
\BIT
\vspace{-1mm}
\item Regularizer / mirror map: $\psi: \rb^d \to \rb \cup \{ +\infty\}$ is a differentiable $\mu$-strongly convex function with respect to some norm $\| \cdot \|$ (which is not in general the $\ell_2$-norm). We will consider in this paper the associated Bregman divergence \citep{bregman1967relaxation} defined as 
$$D_\Psi(\theta,\eta) = \psi(\theta)-\psi(\eta)-\psi'(\eta)^\top(\theta - \eta).$$

\vspace{-1mm}
\item Data:  $x_i \in \rb^{d \times k}$, $\mathcal{Y}_i \subset \rb^k$ are closed convex sets, for $ i \in \{1,\dots,n\}$.
\vspace{-1mm}
\item Feasibility / interpolation regime: we make the assumption that there exists $\theta \in \rb^d$ such that $\psi(\theta) < 
\infty$ and $\forall i \in \{1,\dots,n\}, \ x_i^\top \theta \in \mathcal{Y}_i$.
\vspace{-1mm}
\EIT
 
This is a general formulation that includes any set $\mathcal{K}_i$ like in the introduction (by having $k=d$, $x_i = \idm$, and $\mathcal{Y}_i = \mathcal{K}_i$), with an important particular case $k=1$ (classical linear prediction).

In this paper, we consider primarily the $\ell_p$-norm set-up, where  $\psi(\theta) = \frac{1}{2} \|\theta\|_p^2$ for $p \in (1,2]$, which is $(p-1)$-strongly convex with respect to the $\ell_p$-norm \citep{ball1994sharp,duchi2010composite}. The simplex with the entropy mirror map, which is $1$-strongly convex with respect to the $\ell_1$-norm, could also be considered.

\subsection{From dual guarantees to primal guarantees}

 We can use Fenchel duality to obtain a dual problem for the problem given in Eq.\eqref{eq:just_primal}. We will need the support function $\sigma_{\Y_i}$ of the convex set $\Y_i$, defined as, for $\alpha_i \in \rb^k$ \citep{boyd2004convex}, 
 $$
 \sigma_{\mathcal{Y}_i}(\alpha_i) = \sup_{ y_i \in \Y_i} y_i^\top \alpha_i.
 $$
 We have, by Fenchel duality:
 \begin{eqnarray}
  & &  \min_{\theta \in \rb^d} \psi(\theta) \mbox{ such that }  \forall i \in \{1,\dots,n\}, \ x_i^\top \theta \in \mathcal{Y}_i \label{eq:primal_eq} \\
  & = & \min_{\theta \in \rb^d}  \psi(\theta) + \frac{1}{n} \sum_{i=1}^n \max_{\alpha_i \in \rb^k} \Big\{  \alpha_i^\top x_i^\top \theta - \sigma_{\mathcal{Y}_i}(\alpha_i) \Big\} \notag
\\
&  =&\max_{\forall i, \ \alpha_i \in \rb^k} - \frac{1}{n} \sum_{i=1}^n  \sigma_{\mathcal{Y}_i}(\alpha_i)  - \psi^\star\Big(
-\frac{1}{n} \sum_{i=1}^n x_i \alpha_i \label{eq:dual_eq}
\Big),
 \end{eqnarray}
 with, at optimality, 
 $$\theta^\star = \theta(\alpha^\star) = \nabla \psi^\star \Big(\displaystyle 
- \frac{1}{n} \sum_{i=1}^n x_i \alpha_i
\Big).$$ We denote by $G(\alpha)$ the dual objective function above. With our assumptions of feasibility and strong-convexity of~$\psi$, there is a unique minimizer $\theta^\star \in \rb^d$. The dual problem is bounded from above, and we assume that there exists a maximizer $\alpha^\star \in \rb^{n \times k}$.

In this paper, we will consider dual algorithms to solve the problem disccused earlier in this section, that naturally leads to guarantees on 
$ {\rm gap}(\alpha ) = G(\alpha^\star) - G(\alpha)$. Our first result is to provide some primal guarantees from $\theta(\alpha)$.

\begin{proposition} \label{prop:gen_result_primal_dual}
With our assumption, for any $\alpha \in \rb^{n \times k}$, we have:
$$
D_\Psi( \theta^\star, \theta(\alpha ))  \leqslant {\rm gap}(\alpha ). \label{eq:statement_1_eq}
$$
\end{proposition}
In the above statement, we also assume that $\psi$ is differentiable everywhere, since Bregman divergences are well defined for differentiable functions. However, if we want to relax the above statement for a general function $\psi$ which might not be differentiable, we would need to replace the term $D_\Psi( \theta^\star, \theta(\alpha )) $ in Eq.~\eqref{eq:statement_1_eq} with $\psi(\theta^\star) - \psi(\theta(\alpha ))) - \langle \partial \psi(\theta(\alpha ))), \theta^\star - \theta(\alpha )) \rangle$ where $\partial \psi(\theta(\alpha )))$ is a specific sub-gradient of $\psi$ at point $\theta(\alpha )$. In the proof of Proposition~\ref{prop:gen_result_primal_dual},  we simply use the duility structure of the problem with Fenchel-Young inequality. See the detailed proof in Appendix~\ref{app:primal-dual}.  

This result relates primal rate of convergence and dual rate of convergence, and holds true irrespective of the algorithm used to optimize the dual objective. Using it, we can recover  convergence guarantees for stochastic dual coordinate ascent (SDCA) \citep{shalev2013stochastic} and  accelerated SDCA~\citep{shalev2014accelerated}. Compared to their analysis, our result directly provides rates of convergence from existing results in coordinate descent, but in terms of primal iterates. Details are provided in Appendix~\ref{app:SDCA}.

 \subsection{Randomized coordinate descent}

 Given our relationship between primal iterate sub-optimality and dual sub-optimality gap ${\rm gap}(\alpha)$ for any dual variable $\alpha$ and its corresponding primal variable $\theta(\alpha)$, we can  leverage good existing algorithms on the dual problem. One such well known method is randomized dual coordinate descent, where $\alpha$ and thus $\theta(\alpha)$ will be random. 

 The algorithm is initialized with $\alpha^{(0)}_i=0$ for all $i \in \{1,\dots,n\}$, and at step $t>0$, an index $i(t) \in \{1,\dots,n\}$ is  selected uniformly (for simplicity) at random. The update for proximal randomized coordinate ascent~\citep{richtarik2014iteration} is obtained in the following lemma (whose proof is given in Appendix~\ref{app:coord_descent_update}).
\begin{lemma}\label{lem:RCD_update}
For any uniformly randomly selected coordinate $i(t)$ at time instance $t$, the update for randomized proximal coordinate ascent  is equal to
\begin{align*} 
\alpha_{i(t)} = \alpha^{(t-1)}_{i(t)} +  \frac{n}{L_{i(t)}} x_{i(t)} ^\top \theta(\alpha^{(t-1)})  -  \frac{n}{L_{i(t)}}\Pi_{\Y_i} \Big( \frac{L_{i(t)}}{n}  \alpha^{(t-1)}_{i(t)} +  x_{i(t)} ^\top \theta(\alpha^{(t-1)}) \Big),
\end{align*}
where $\Pi_{\Y_i}$ is the orthogonal projection on $\Y_i$, and $L_{i}$ is equal to
 $ \displaystyle
 L_{i} = \frac{1}{\mu} \| x_i\|_{ 2 \to \star }^2 = \frac{1}{\mu} \sup_{\| \beta_i\|_2 = 1} \| x_i \beta_i\|_\star^2$.
\end{lemma}
Here, we implicitly assume that the individual projections on convex set $\mathcal{Y}_i$ for all $i \in \{1,\cdots,n \}$  are easy to compute, leading to Algorithm~\ref{alg:random_p_general}. For uniformly random selection of the datapoint $x_{i(t)}$ at time $t$, $L_i(t)$ can simply be replaced by $\max_i L_i$ in the algorithm. 


\begin{wrapfigure}{L}{0.5\textwidth}
\begin{minipage}{0.5\textwidth}
 \begin{algorithm}[H]
\DontPrintSemicolon

  \KwInput{$\alpha_0$, $\theta_0\gets \theta(\alpha_0)$ and $\vx_i,\mathcal{Y}_i$ for $i \in[n]$ .}

  \KwOutput{$\theta_{T+1}$ and $\alpha_{T+1}$}
  
\For{$t\gets1$ \KwTo $T$ }{
    Choose $i_t \in \{1,2,\cdots,n \}$ randomly. \\
    $\beta_{\text{(prev)}} =\alpha_{i(t)}^{(t-1)} $ \\
    $\zeta_t = \Pi_{\Y_i} \Big( \frac{L_{i(t)}}{n}  \alpha^{(t-1)}_{i(t)} +  x_{i(t)} ^\top \theta_{t-1} \Big) .$ \\
    $\alpha_{i(t)} = \alpha^{(t-1)}_{i(t)} +  \frac{n}{L_{i(t)}} x_{i(t)} ^\top \theta_{t-1}- \frac{n}{L_{i(t)}}\zeta_t.$ \\
    $\Delta_\beta = \alpha_{i(t)} -\beta_{\text{(prev)}}  .$ \\
    Update $\theta_{t+1} \gets \theta(\alpha_{t+1}) ~~~~~\{\text{Use~}\Delta_\beta, x_{i(t)}\}.$ 
    }
\caption{ Proximal Random Coordinate Ascent}\label{alg:random_p_general}
\end{algorithm}
\vspace{-1cm}
\end{minipage}
\end{wrapfigure}

  
Proximal randomized coordinate descent is a well studied problem~\citep{nesterov2017efficiency,richtarik2014iteration}, and has a known rate of convergence for smooth objective functions. The set of optimal solutions of the dual problem in Equation~\eqref{eq:dual_eq} is denoted by $A^\star$ and $\alpha^\star$ is an element of it.  Define, 
\begin{align*}
\mathcal{R}(\alpha) = \max_y \max_{\alpha^\star \in A^\star } \left\{ \| y - \alpha^\star \|~ :~ G(y) \geq G(\alpha)  \right\}.
\end{align*}

Since we assumed that $\psi$ is $\mu$-strongly convex,   $\psi^\star$ is $(\frac{1}{\mu})$-smooth, and we get
\begin{align}
\E \Big[D_\Psi( \theta^\star, \theta(\alpha^{(t)})) \Big] \leqslant \E\big[{\rm gap}(\alpha^{(t)})\Big] \notag \\
\leqslant \frac{ \max_{i} L_i}{t} \frac{\max\{\| \alpha^\star\|^2, \mathcal{R}(0)^2\}}{n}, \label{eq:rcd_rate}
\end{align}
where $L_i$ is defined in Lemma~\ref{lem:RCD_update}. 
The convergence rate given in Eq.~\eqref{eq:rcd_rate} can further be improved with non-uniform sampling based on the values $L_i$, and then $\max_{i}L_i$ can be replaced by $\frac{1}{n}\sum_{i=1}^nL_i$~\citep{richtarik2014iteration}. However, taking inspirations from \citep{cutkosky2019anytime,kavis2019unixgrad} the convergence for averaged iterate of coordinate descent when $\mathcal{Y}_i $ is a singleton set for all $i$ can be obtained which only depends on $\|\alpha^\star \|$.
 
\subsection{Relationship to least-squares}
We now discuss an important case of the above formulation when $\mathcal{Y}_i $ is a singleton set, i.e., $\mathcal{Y}_i = \{y_i\}$. This problem has been addressed recently by~\citet{calatroni2019accelerated} and we recover it as a special case of our general formulation. 

We will make a link with  least-squares in the interpolation regime, which can be written as a finite sum objective as follows,
 \begin{align}
 \min \left[\frac{1}{2n} \sum_{i=1}^n \| y_i - x_i^\top \theta \|_2^2
 = \frac{1}{2n} \sum_{i=1}^n d( x_i^\top \theta, \mathcal{Y}_i)^2\right]. \label{eq:least_square}
 \end{align}
 It turns out that primal stochastic mirror descent with constant step-size applied to Eq.~\eqref{eq:least_square} and our formulation provided in \Cref{eq:primal_eq,eq:dual_eq} are equivalent, as we now show.
 \begin{lemma}\label{lem:mirro_coord_dual}
 Consider the stochastic mirror descent updates using the mirror map $\psi$ for the least-squares problem provided in Eq.~\eqref{eq:least_square}. Then, the corresponding stochastic mirror descent updates converges to minimum $\psi$ solution. 
 \end{lemma}
 \begin{proof}
 Consider the primal-dual formulation given in Eq.~\eqref{eq:primal_eq} and Eq.~\eqref{eq:dual_eq}, with  $\mathcal{Y}_i = \{y_i\}$. The randomized dual coordinate ascent has the following update rule:
\begin{align}
\alpha^{(t)}_{i(t)} =  \alpha^{(t-1)}_{i(t)} +  \frac{n}{L_{i(t)}} (x_{i(t)} ^\top \theta(\alpha^{(t-1)})   - y_{i(t)}  ). \label{eq:sto_mirror_least_square}
\end{align}
 From the first order optimality condition, the update in Eq.~\eqref{eq:sto_mirror_least_square} translates into, with $\theta^{(t)} = \theta(\alpha^{(t)})$,
 $$
 \psi'(\theta^{(t)}) =  \psi'(\theta^{(t-1)}) - \frac{1}{L_{i(t)}}  x_{i(t)}  ( x_{i(t)} ^\top \theta(\alpha^{(t-1)}) - y_{i(t)} ),
 $$
 which is exactly stochastic mirror descent on the least-squares objective  with mirror map $\psi$. Hence the result.
  \end{proof}
  The rate of convergence can be obtained by the use of  Eq.~\eqref{eq:rcd_rate}.

  \paragraph{General case (beyond singletons).} For any set $\Y_i$, if $\alpha_{i(t)}^{(t-1)}=0$, for example,  if $i(t)$ has never been selected, then, by Moreau's identity, we also get a stochastic mirror descent step for $\frac{1}{2n} \sum_{i=1}^n d( x_i^\top \theta, \mathcal{Y}_i)^2$. However, this is not true anymore when an index is selected twice.

\subsection{Accelerated coordinate descent}

In the previous sections, we discussed randomized coordinate dual ascent to optimize the problem in Eq.~\eqref{eq:primal_eq}. We can also consider accelerated proximal randomized coordinate ascent~\citep{lin2015accelerated,hendrikx2019accelerated,allen2016even}. For our problem, it leads to:
\begin{align}
\E\Big[D_\Psi( \theta^\star, \theta(\alpha^{(t)})) \Big] \leqslant \E\Big[{\rm gap}(\alpha^{(t)})\Big]
\leqslant  \frac{ 4  \max_{i} L_i}{t^2} \left\{ \frac{G(\alpha^\star) - G(0)}{ \max_{i} L_i} + \frac{1}{2}\| \alpha^\star\|^2\right\}. \label{eq:acc_rcd_rate}
\end{align}
We will use the bound in Eq.~\eqref{eq:acc_rcd_rate} to analyze the general perceptron in the next section. 
We also provide the proximal accelerated randomized coordinate ascent  algorithm \citep{lin2015accelerated,hendrikx2019accelerated} with uniformly random sampling of coordinates to optimize the dual objective of $\ell_p$-perceptron. However, the algorithm can easily be updated for the general case of Eqs.~\eqref{eq:primal_eq} and \eqref{eq:dual_eq}. 

\vspace{-2mm}
\begin{figure*}[!h]
\centering
 \begin{minipage}[t]{0.51\textwidth}
 \centering
\begin{algorithm}[H]
\DontPrintSemicolon

  \KwInput{$\alpha_0$, $\theta_0\gets \theta(\alpha_0)$, $x_i$ for $i \in[n]$ and $\mu=0$.}
  \textbf{Initialize: } $z_0 \gets \alpha_0 $, $\theta_{z_0}\gets \theta_0$, $v_0 \gets \alpha_0 $ and $\gamma_0\gets \frac{1}{n}.$ \;

  \KwOutput{$\theta_{T+1}$ and $\alpha_{T+1}$}
  
\For{$t\gets0$ \KwTo $T$ }{
    Choose $i_t \in \{1,2,\cdots,n \}$ randomly. \\
    $r_t = 1 - \theta_{z_t}^\top x_{i_t}$ \\
    $\alpha_{t+1} = u_{t+1} = \alpha_t + \frac{r_t}{n \gamma_t L_{i_t}} .$ \\
    $\alpha^{(t+1)}_{i(t)} = \max(\alpha^{(t+1)}_{i(t)},0).$ \\
    Update $\theta_{t+1}\gets \theta(\alpha_{t+1}).$~~~~~~~~~~~~~(Algorithm~\ref{alg:theta_update }) \\
    $\gamma_{t+1} = \frac{1}{2}\left( \sqrt{\gamma_t^4 + 4\gamma_t^2} - \gamma_t^2 \right).$\\
    $v_{t+1} = z_t + n\gamma_t (\alpha_{t+1} - \alpha_{t}).$ \\
    $z_{t+1} = (1-\gamma_{t+1})v_{t+1} + \gamma_{t+1} \alpha_{t+1}.$ \\
    Update $\theta_{z_{t+1}}\gets \theta(z_{t+1}).$~~~~~~~~~~~~~(Algorithm~\ref{alg:theta_z_t_update })
    }
\caption{ Accelerated Proximal Coordinate Ascent (Dual Perceptron)  \citep{lin2015accelerated,hendrikx2019accelerated}}\label{alg:accel_p_general}
\end{algorithm}
\end{minipage}
\hfill
\begin{minipage}[t]{0.43\textwidth}
\centering
\begin{algorithm}[H]
\DontPrintSemicolon

  \KwInput{$x_{i_t}$,$\alpha_{t+1}$ , $X^\top\alpha_t$, $\alpha_t$ and $i_t$ .}

  \KwOutput{$\theta_{t+1}$ and $X^\top\alpha_{t+1}$}
  $~~ X^\top\alpha_{t+1} = X^\top\alpha_{t} + (\alpha^{(t+1)}_{i(t)} - \alpha^{(t)}_{i(t)})x_{i_t}.$ \\
  $\qquad \qquad$ Compute $\theta_{t+1}$ from $X^\top\alpha_{t+1}.$
  
\caption{ Update $\vtheta_{t+1}$}\label{alg:theta_update }
\end{algorithm}
\begin{algorithm}[H]
\DontPrintSemicolon

  \KwInput{$x_{i_t}$, $\alpha_{t+1}$,   $X^\top\alpha_t$, $X^\top\alpha_{t+1}$, $X^\top z_{t}$, $\alpha_t$, $\gamma_t$,  $\gamma_{t+1}$ .}

  \KwOutput{$\vtheta_{\vz_{t+1}}$ and $X^\top z_{t+1}$}
 $ ~~X^\top v_{t+1} = X^\top z_t + n\gamma_t X^\top(\alpha_{t+1} - \alpha_t).$ \\
  $\qquad X^\top z_{t+1} = (1 - \gamma_{t+1}) X^\top v_{t+1}   +\gamma_{t+1} X^\top \alpha_{t+1} .$ \\
  $\qquad \qquad$ Compute $\theta_{z_{t+1}}$ from $X^\top z_{t+1}.$  
\caption{ Update $\vtheta_{\vz_{t+1}}$}\label{alg:theta_z_t_update }
\end{algorithm}
\end{minipage}
\end{figure*}

\subsection{Baseline: Primal Mirror Descent} \label{subsec:sharan_extension}
We will compare our dual algorithms to existing primal algorithms. They correspond to the minimization of 
\begin{equation}
    \label{eq:finite_sum}
F(\theta) =  \frac{1}{2n} \sum_{i=1}^n d( x_i^\top \theta, \mathcal{Y}_i)^2.
\end{equation}
\citet{vaswani2018fast} showed convergence of stochastic gradient descent for this problem. We extend their results to all mirror maps.
Mirror descent with the mirror map $\psi $ selects $i(t)$ at random and the iteration update is 
\begin{align}
\psi'(\theta^{(t)}) &= \psi'(\theta^{(t-1)})  - \gamma x_{i(t)} ( \Pi_{\mathcal{Y}_i}(  x_{i(t)}^\top \theta^{(t-1)} ) - x_{i(t)}^\top \theta^{(t-1)} ) . \label{eq:mirror_sharan}
\end{align}
Note that we have already encountered it in Lemma~\ref{lem:mirro_coord_dual}, for least-squares regression, where we provided a convergence rate on the final iterate.

In Theorem~\ref{thm:sharan_mirror_desc} below, we prove  an ${O}\left({1}/{t}\right)$ convergence rate for stochastic mirror descent update with mirror map $\psi$, for a constant step-size and the average iterate, directly extending the result of~\citet{vaswani2018fast} to all mirror maps.
\begin{theorem}\label{thm:sharan_mirror_desc}
Consider the stochastic mirror descent  update in Eq.~\eqref{eq:mirror_sharan} for the optimization problem in Eq.~\eqref{eq:finite_sum} with $\gamma = \mu / \sup_{i} \| x_i\|_{ 2 \to \star }^2$, the expected optimization error after~$t$  iterations the for averaged iterate $\bar{\theta}_t$ behaves as,
$$0 \leqslant \E [ F(\bar{\theta}^{(t)}) ] \leqslant \frac{\max_{i}  L_i}{  t} \psi(\theta^\star). $$
\end{theorem}
We provide the proof in Appendix~\ref{app:mirror_descent_sharan}. The result is also applicable to general expectations and any form of convex objectives in the interpolation regime. We use this extension as one of our baseline in our experiments. In practice, as mentioned earlier, the update for mirror descent in Eq.~\eqref{eq:mirror_sharan} is similar to randomized dual coordinate ascent update in Lemma~\ref{lem:RCD_update}, in particular in early iterations (and not surprisingly, they behave similarly). Note here the difference in guarantees for the final iterates (which we get through a dual analysis) and the guarantees for the averaged iterate (which we get through a primal analysis).

\section{$\ell_p$-perceptrons} \label{sec:lp_percept}
So far, we have discussed very general formulations for optimization problems in the interpolation regime. In this section, we  discuss a specific problem which is widely used for linear binary classification, known as the perceptron algorithm, which is guaranteed to converge for linearly separable data. Here, we view the generalized $\ell_p$-norm perceptron algorithm from the lens of our primal-dual formulation. 

We consider $(x_i,y_i) \in \rb^d \times \{-1,1\}$ for $i \in \{1, \cdots, n\}$, and the problem of minimizing $\psi(\theta)$ such that $\forall i, y_i x_i^\top \theta \geqslant 1$, which can be written as $\tilde{x}_i^\top \theta \geqslant 1$, where $\tilde{x}_i = y_i x_i$ for all $i \in \{1, \dots, n\}$. For this section, we will be limiting ourselves to $\psi(\theta) = \frac{1}{2}\| \theta\|_p^2$ for $p\in(1,2]$. We know that $\psi(\theta) = \frac{1}{2} \|\theta\|_p^2$ for $p \in (1,2]$ is $(p-1)$-strongly convex with respect to the $\ell_p$-norm. In this section, we denote $X\in \mathbb{R}^{n\times d}$ the data matrix $X = (\tilde{x}_1^\top ;\tilde{x}_2^\top ; \cdots ; \tilde{x}_n^\top)$. Our generic optimization problem from \eq{just_primal} turns into:
\begin{equation}
\min_{ \theta \in \rb^d} \frac{1}{2} \| \theta \|_p^2 \mbox{ such that } X\theta \geqslant 1, \label{eq:percept_form}
\end{equation}
The dual problem is here 
\begin{equation}
  \max_{\alpha \in \rb_+^n}  - \frac{1}{2} \bigg\| \frac{-1}{n} \sum_{i=1}^n x_i \alpha_i \bigg\|_q^2 + \frac{1}{n}\sum_{ i=1}^n\alpha_i   \label{eq:percet_dual} ,
\end{equation}
 where $\| \cdot\|_q$ is dual norm of $\| \cdot\|_p$, with $1/p+1/q=1$.  At optimality, $\theta $  can be obtained from $X^\top \alpha$ as 
 $$
 \theta_j = \frac{1}{n}\| X^\top \alpha \|_q^{2-q}  (X^\top \alpha)_j^{q-1}
  ,$$
   where we define $u^{q-1} = |u|^{q-1}  {\rm sign}(u)$.
   
    The function $ \alpha \mapsto \frac{1}{2} \| X^\top \alpha \|_q^2$ is smooth, and the regular smoothness constant with respect to the $i$-th variable which is less than 
 $
 L_i = \frac{1}{p-1} \| x_i \|_q^2.$
 We can apply here the results from Proposition~\ref{prop:gen_result_primal_dual} to get the convergence in primal iterates for the the general $\ell_p$-norm perceptron formulation in Eq.~\eqref{eq:percept_form}, while optimizing the dual function via accelerated coordinate ascent in Eq.~\eqref{eq:percet_dual}.
 \begin{corollary} \label{cor:percept_1}
 For the generalized $\ell_p$-norm  perceptron described in our primal-dual framework in Equations~\eqref{eq:percept_form} and \eqref{eq:percet_dual}, we have
$ \displaystyle
    \E\Big[ \| \theta(\alpha)- \theta^\star\|_p \Big]\leq \sqrt{\frac{2 \E[{\rm gap}(\alpha)]}{p-1}}.
$
 \end{corollary}
 \begin{proof}
 The result comes from the application of Proposition~\ref{prop:gen_result_primal_dual} in the generalized $\ell_p$-norm perceptron from setting Eq.~\eqref{eq:percept_form}, with   $D_{\frac{1}{2}\|\cdot \|_p^2}(\theta^\star, \theta) \geq \frac{p-1}{2}\| \theta - \theta^\star\|_p^2$.
 \end{proof}
If we use accelerated randomized coordinate descent to optimize dual objective given in Eq.~\eqref{eq:percet_dual}, then after $t$ number of iterations, we get:
\begin{align}
    \E\Big[\| \theta_t- \theta^\star\|_p\Big] \leq \frac{    2\sqrt{2}      \max_i \|x_i \|_q}{\sqrt{(p-1)}t} \sqrt{\frac{G(\alpha^\star) - G(0)}{ \max_i \|x_i \|_q} + \frac{1}{2}\| \alpha^\star\|^2}, \label{eq:acc_percet_p}
\end{align}
where $\theta_t = \theta(\alpha_t).$ 

\paragraph{Mistake bound.} Since, we have the bound on the distance between primal iterate to its optimum, we can simply derive  the mistake bound for our algorithm which we prove in  Appendix~\ref{app:l_p_percept}.
\begin{lemma}\label{lem:mistake_l_p}
For the generalized $\ell_p$-norm  perceptron described in our primal-dual framework in Equations~\eqref{eq:percept_form} and \eqref{eq:percet_dual},  we make no mistakes on training data on average after 
$$ t >\frac{2\sqrt{2}R^2 }{\sqrt{p-1}  } \sqrt{\frac{G(\alpha^\star) - G(0)}{ R} + \frac{1}{2}\| \alpha^\star\|^2} $$ steps where $R = \max_i \| x_i\|_q$ and $\| \cdot\|_q$ is the dual norm of $\| \cdot\|_p$. 
\end{lemma}
The accelerated coordinate descent algorithm to solve the $\ell_p$-perceptron is given in Algorithm~\ref{alg:accel_p_general}. More details about the relationship between primal and dual variables, as well as dual ascent update for  random coordinate descent  for general $\ell_p$-norm perceptron, e.g., the dual problem in Eq.~\eqref{eq:percet_dual}, is given in Appendix~\ref{app:l_p_percept}. Mistake bounds for the classical $\ell_p$-perceptron are also recalled in  Appendix~\ref{app:l_p_percept}.

\paragraph{Baseline: primal mirror descent.} We consider the finite sum minimization with stochastic mirror descent update and mirror map $\psi = \frac{1}{2}\| \cdot\|_p^2$ as discussed in Section~\ref{subsec:sharan_extension}, that is, the finite sum minimization in Eq.~\eqref{eq:finite_sum} with $f_i(\theta) =\frac{1}{2} (1 - \theta^\top x_i)_{+}^2. $
\begin{corollary} \label{cor:percept_sharna}
Consider the finite sum minimization of $f(\theta) = \frac{1}{2n}\sum_{i=1}^n (1 - \theta^\top x_i)_{+}^2$ via stochastic mirror descent with mirror map $\psi(\cdot) = \frac{1}{2}\| \cdot\|_p^2$,  then on average, the proportion of mistakes on the training set is less than $\sqrt{\frac{\| \theta^\star\|_p^2 R^2}{(p-1)t}}$ where $R = \max_i \| x_i\|_q$.
\end{corollary}
\begin{proof}
The proof comes directly from Theorem~\ref{thm:sharan_mirror_desc} and from the fact that the proportion of mistakes on the training set is less than the square root of the excess risk.
\end{proof}
Similar bounds on the proportion of mistakes can also be obtained while optimizing $f(\theta) = \frac{1}{n}\sum_{i = 1}^n (1 - x_i^\top \theta)_+$ via stochastic mirror descent with mirror map $\frac{1}{2}\| \cdot\|_p^2$. However, while tuning the step size, it requires the knowledge of $\| \theta^\star\|_p$, hence we do not include it in our base line. 

We can compare the minimum number of iterations required to achieve no further mistakes while training in Lemma~\ref{lem:mistake_l_p} and Corollary~\ref{cor:percept_sharna} to get the conditions on optimal primal and dual optimal variables under which our method (which has a better dependence in the number of iterations $t$) performs better than the baseline. We discuss these in the Appendix~\ref{app:l_p_percept}. In our empirical evaluationin Section~\ref{sec:expts}, dual accelerate coordinate ascent significantly outperforms primal mirror descent.

\paragraph{Special Case of $\ell_1$-perceptron.} Our goal in this specific case is to solve the following sparse problem,
\begin{align}
\theta_0 = \argmin_{ \theta \in \rb^d} \frac{1}{2} \| \theta \|_1^2
\mbox{ such that } X\theta \geqslant 1. \label{eq:sparse_perceptron}
\end{align}
 $\|\cdot\|_1$ is not strongly convex, hence we can not fit this problem to our formulation. However, following~\citet{duchi2010composite},  we solve the problem in \eqref{eq:percept_form} with $p = 1+\frac{1}{\log d}$ where $d$ is the dimension.

\begin{figure*}[h!] 
\centering
\begin{subfigure}[t]{0.47\textwidth}
  \centering
  \includegraphics[width=7cm,height=5.2cm]{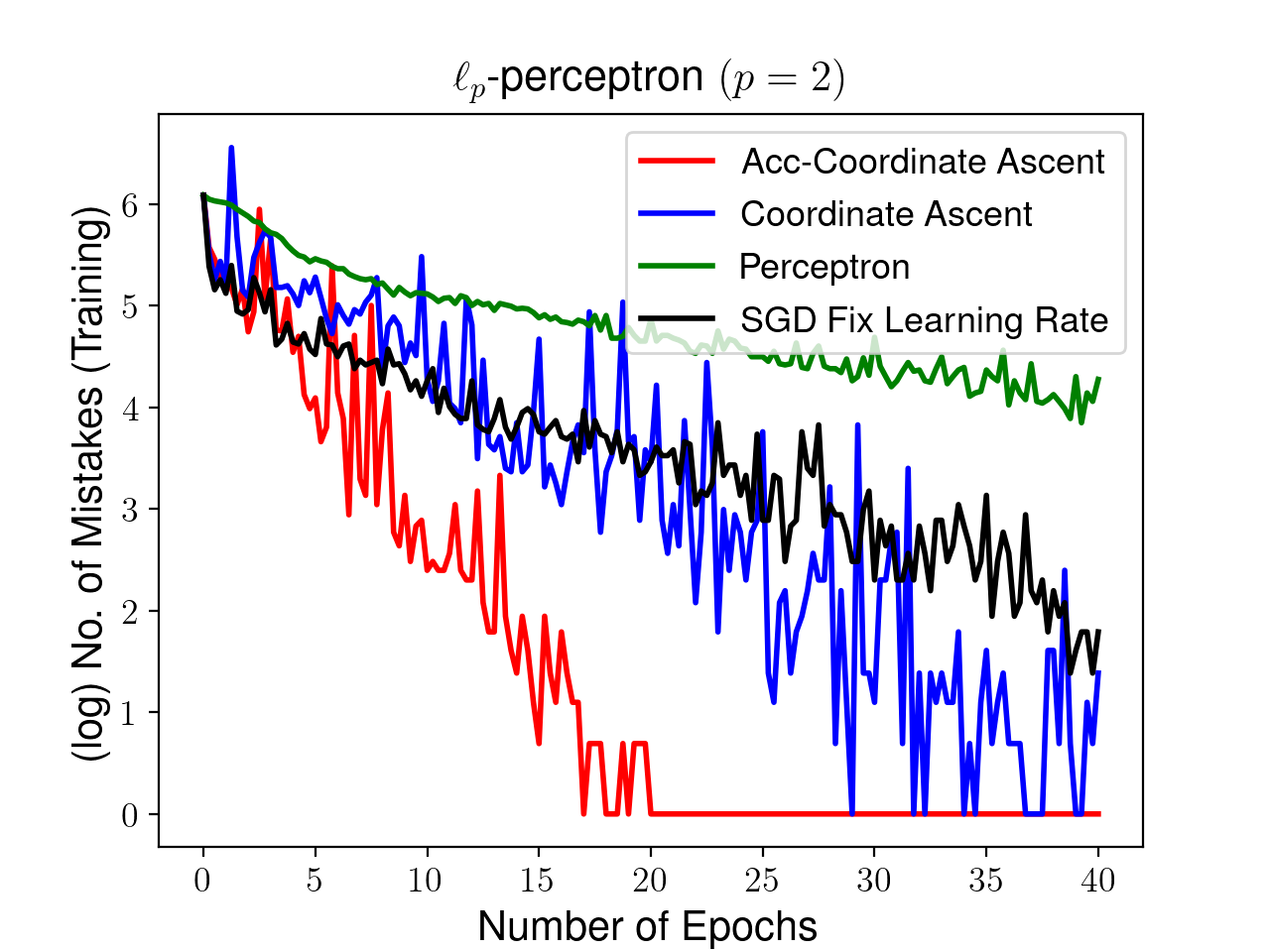}
  \caption{Number of mistakes on the training test (in log scale).}
  \label{fig:l2_train}
\end{subfigure}%
~
\begin{subfigure}[t]{0.47\textwidth}
  \centering
  \includegraphics[width=7cm,height=5.2cm]{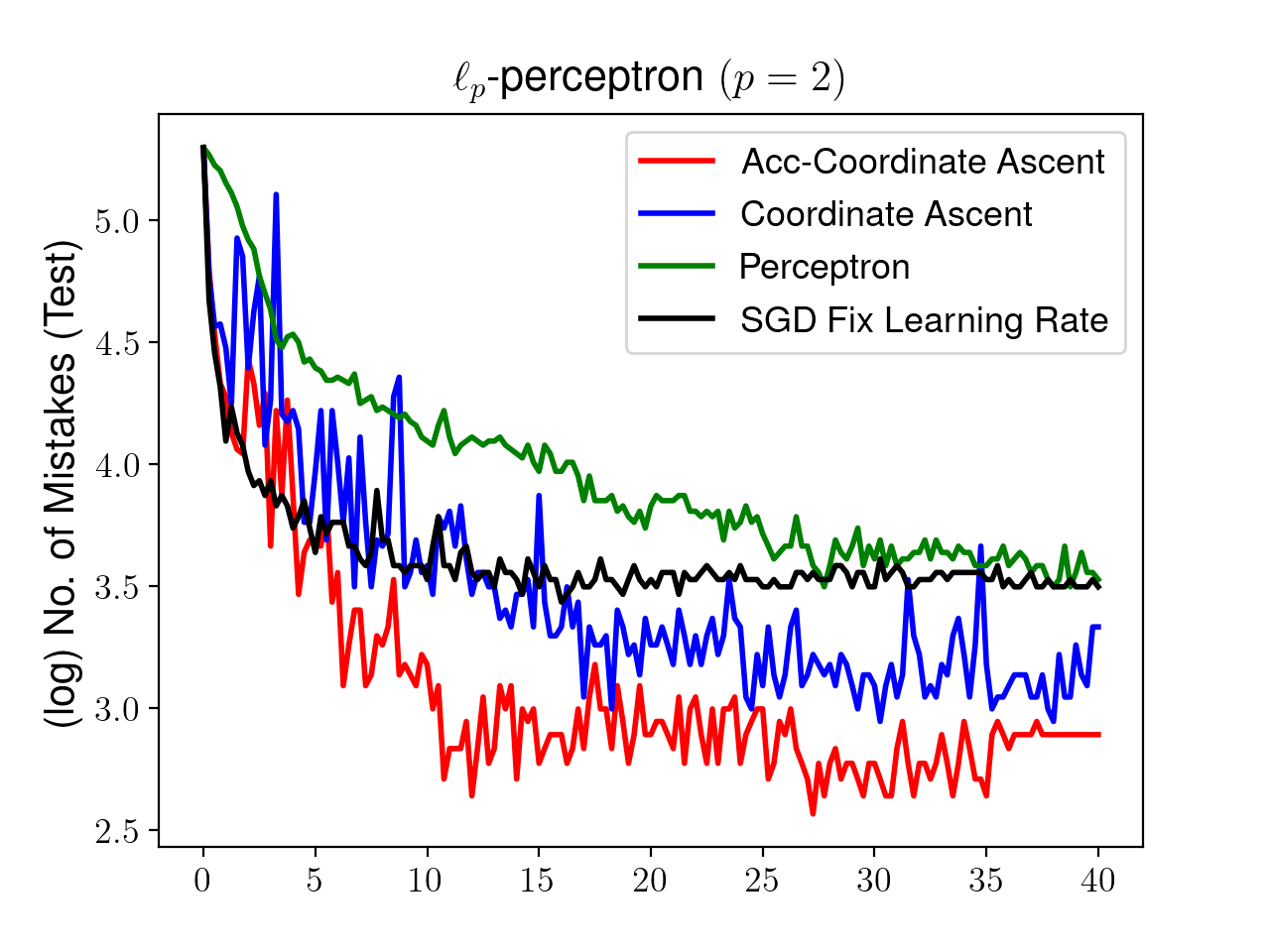}
  \caption{Number of mistakes on the test  (in log scale).}
  \label{fig:l2_test}
\end{subfigure}
\vspace{-1mm}
\caption{Experimental results for $\ell_2$-perceptron}
 \label{fig:l_2_percept}
\end{figure*}

\begin{figure*}[h!] 
\centering
\begin{subfigure}[t]{0.47\textwidth}
  \centering
  \includegraphics[width=7cm,height=5.2cm]{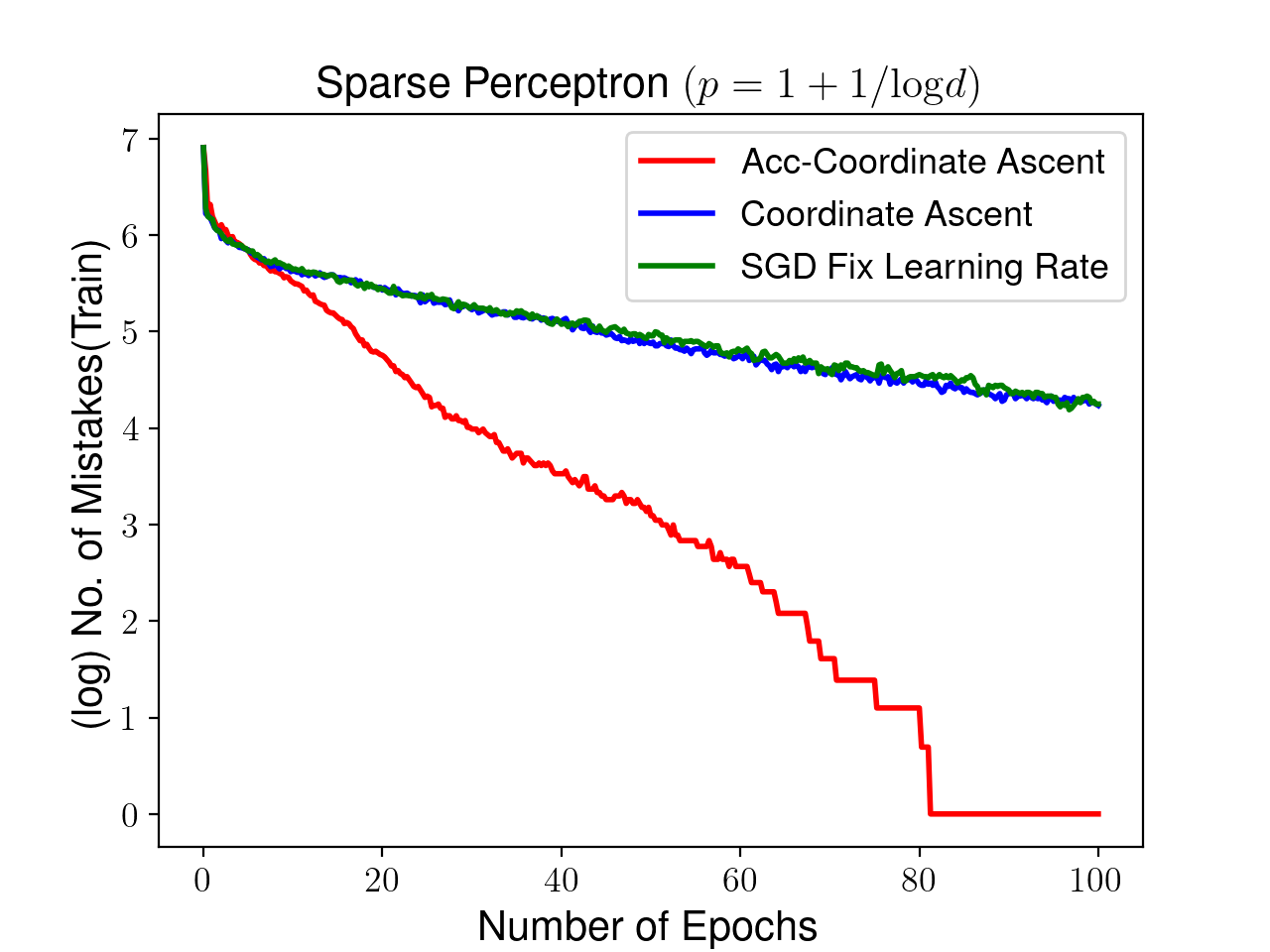}
  \caption{Number of mistakes on the training (in log scale).}
  \label{fig:l1_train}
\end{subfigure}%
~
\begin{subfigure}[t]{0.47\textwidth}
  \centering
  \includegraphics[width=7cm,height=5.2cm]{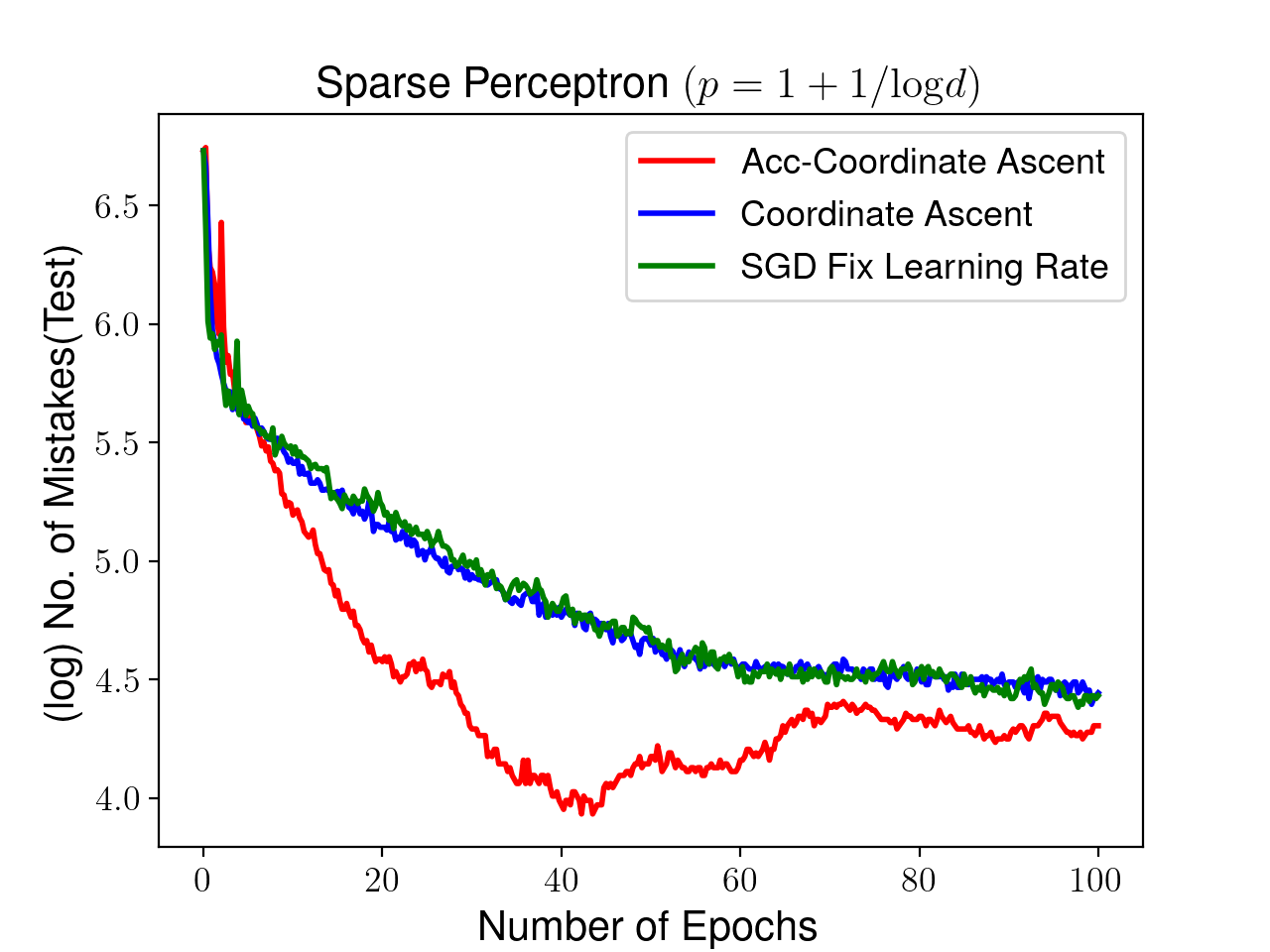}
  \caption{Number of mistakes on the test (in log scale).}
  \label{fig:l1_test}
\end{subfigure}%
\vspace{-1mm}
\caption{Experimental results for sparse perceptron.}
 \label{fig:l_1_percept}
\end{figure*}

 \section{Experiments} \label{sec:expts}
 In this section, we provide empirical evaluation for the methods discussed in this paper with the $\ell_p$-perceptron. We generate data from a  Gaussian distribution in dimension $d=2000$, which we describe below.  We consider two settings of $p$ for our experiments, $p=2 $ which is usual perceptron, and $p = 1+ \frac{1}{\log d}$, which is the sparse perceptron setting. 
 
 \paragraph{Data generation.} We generate $n=1000$ inputs $x_i \in \R^d$, $i=1,\dots,n$ with $d=2000$ from a Gaussian distribution centered at 0 and covariance matrix $\Sigma$ which is a diagonal matrix. Similarly, we generate a random $d=2000$ prediction vector $\theta$ sampled again from the normal distribution.
 
 For $\ell_2$-perceptron, the $i$-th eigenvalue for $\Sigma$ is $1/{i}^{3/2}$ and for sparse perceptron $i$-th eigenvalue for $\Sigma$, is $1/{i}$. We compute the prediction vector $y_i$ for $x_i$ as follows, $y_i = {\rm sign}(x_i^\top \theta + b)$ where we fix $b =0.005$. We also remove those pair of $(x_i,y_i)$ from the data for which we have $x_i^\top \theta + b \leq 0.1$. 
We generate $1000$ train examples and $1000$ test examples for both  settings. 
 For the sparse perceptron case, we make the prediction vector $\theta$ sparse by randomly choosing $50$ entries to be non zero. We then compute the prediction vector similar to the $\ell_p$-perceptron case, $y_i = {\rm sign}(x_i^\top \theta + b)$ where we fix $b =0.005$ and remove those pair of $(x_i,y_i)$ from the data for which we have, $x_i^\top \theta + b \leq 0.1$.
 
 \paragraph{Baseline.} For the $\ell_2$-perceptron, we compare accelerated coordinate descent and randomized coordinate descent with the perceptron and primal SGD~\citep{vaswani2018fast}. For the sparse perceptron, we compare the accelerated coordinate descent and randomized coordinate descent with extension of primal SGD to stochastic mirror descent case (discussed in section~\ref{subsec:sharan_extension} with $f_i = \frac{1}{2} (1-x_i^\top\theta)_+^2$) with mirror map $\psi(\cdot) =\frac{1}{2}\| \cdot\|_p^2$ where $p = 1+ \frac{1}{\log d}$. Note that we compare to non-averaged SGD (for which we provide a new proof), which works significantly better than averaged SGD.
 
 Comparisons for the $\ell_2$-perceptron and sparse perceptron are given in Figures \ref{fig:l_2_percept} and \ref{fig:l_1_percept} respectively. 
 
 We can make the following observations:
 \begin{itemize}
     \item[(a)] From both the training plots (Figure~\ref{fig:l2_train} and Figure~\ref{fig:l1_train}), it is clear that we gain  significantly in training performance over primal SGD and the perceptron if  we optimize the dual with accelerated randomized coordinate ascent method, which supports our theoretical claims made in Section~\ref{sec:lp_percept}.
     
     \item[(b)] For testing errors, we also see gains for our accelerated perceptron, which is not supported by theoretical arguments. This gives motivation to further study this algorithm for general expectations.
     
     \item[(c)] Note that in the semi-log plots, we observe an affine behavior of the training errors, highlighting exponential convergence. This can be explained by a strongly convex dual problem (since the matrix $XX^\top$ is invertible), and could be quantified using usual convergence rates for coordinate ascent for strongly-convex objectives.
 \end{itemize}{}

 \section{Conclusion} \label{sec:conclude}
 In this paper, we proposed algorithms that are explicitly regularizing solutions of an interpolation problem. This is done through a dual approach, and, with acceleration, it improves over existing algorithms. Several natural questions are worth exploring: (1) Can we explicitly characterize linear convergence in the dual (like observed in experiments), with or without regularization? (2) How are our algorithms performing beyond the interpolation regime, where the dual become unbounded but some primal information can typically be recovered in Dykstra-style algorithms~\citep{bauschke2015projection}? (3)
Can we extend our approach to  saddle-point formulations such as proposed by~\citet{kundu2018convex}? Can we prove any improvement in  the general population regime, where we aim at bounds on testing data?

\subsection*{Acknowledgements}
This work was funded in part by the French government under management of Agence Nationale
de la Recherche as part of the ``Investissements d'avenir'' program, reference ANR-19-P3IA-0001
(PRAIRIE 3IA Institute). We also acknowledge support from the European Research Council (grant
SEQUOIA 724063).

\bibliography{ml}
\bibliographystyle{icml2020}
 \clearpage
 \appendix
 
\section{Primal-Dual Structure}\label{app:primal-dual}
Apart from the notations discussed in the main paper, we would further use the following notation for data matrix $X \in \mathbb{R}^{n\times d}$ such that  $X = \left[x_1^\top ; \cdots ; x_n^\top \right]$.  We consider the following general primal and its corresponding dual problem  which appear very frequently in machine learning domain. 
\begin{align}
&\min_{\theta \in \mathbb{R}^d } ~ \Bigg[ \mathcal{O}_{P}(\theta):=\psi(\theta)  + \frac{1}{n}\sum_{i=1}^{n} \phi_i(x_i^\top\theta)   \Bigg] \label{eq:primal_gen_erm} \\
&\max_{\alpha \in \mathbb{R}^n}~ \Bigg[ \mathcal{O}_{D}(\alpha):= - \psi^*\left(- \frac{1}{n} X^\top\alpha \right) - \frac{1}{n}\sum_{i=1}^n \phi_i^* (\alpha_i) \Bigg]. \label{eq:dual_gen_erm}
\end{align}
Here, we assume that $\phi:\mathbb{R}^d\rightarrow \mathbb{R}$ and $\psi:\mathbb{R}^d\rightarrow \mathbb{R}$ are smooth convex function for all $i$.
 We have the following first order optimality conditions for the equivalent problems given in Equations~\eqref{eq:primal_gen_erm} and \eqref{eq:dual_gen_erm}:
\begin{equation} \label{eq:optimality_cond}
  \begin{split}
    &x_i^\top \theta \in \partial \phi_i^*(\alpha_i),\\
    &\theta \in \partial \psi^*\left(-\frac{1}{n}X^\top \alpha\right), 
  \end{split}
\quad \text{and}\quad
  \begin{split}
     &\alpha_i \in \partial \phi_i(x_i^\top \theta),\\
    &-\frac{1}{n}X^\top \alpha  \in \partial \psi(\theta).
  \end{split}
\end{equation}

From the duality, $\theta(\alpha) = \partial \psi^*\left( -\frac{1}{n}\sum_{i=1}^n \alpha_i  x_i\right)$. We can 
 recall Fenchel's Inequality: For any convex function $f$, the inequality $f(x)+f^*(\theta) \geq x^\top \theta$ holds for all $x \in {\rm dom}(f)$ and $\theta\in {\rm dom}(f^*)$. Equality holds if the following is satisfied $\theta \in \partial f(x)$. 

From Fenchel's inequality, we have:

\begin{proposition} \label{prop:dual_sub_opt_divergence_gen}
Consider the general primal dual problem given in equations~\eqref{eq:primal_gen_erm} and \eqref{eq:dual_gen_erm},  dual sub-optimlaity gap $gap(\alpha) = [\mathcal{O}_{D}(\alpha^\star) - \mathcal{O}_{D}(\alpha)]$ at some $\alpha$ provides the upper bound on the Bregman divergence of $\psi$ between $\theta^\star$ and $\theta(\alpha)$ \textit{i.e.} $D_\Psi( \theta^\star, \theta(\alpha ))  \leqslant {\rm gap}(\alpha ).$
\end{proposition}
\begin{proof}
 The Bregman divergence with respect to mirror map $\psi$ is 
\begin{align*}
D_\Psi(x,y) = \psi(x) - \psi(y) - \langle \nabla \psi(y), x - y \rangle.
\end{align*}
Now, we have: 
\begin{align}
\begin{split}
{\rm gap}(\alpha) &=  - \psi^*\left(-\frac{1}{n} X^\top\alpha^\star\right)  +  \psi^*\left(- \frac{1}{n} X^\top\alpha \right) - \frac{1}{n}\sum_{i =1}^n \phi_i^* (\alpha_i^\star) +\frac{1}{n} \sum_{i =1}^n \phi_i^* (\alpha_i)  .
\end{split}
\end{align}
In the proof we would again use Fenchel's inequality which we used in the proof of previous theorem. 
From the optimality condition, we know that $-\frac{1}{n}X^\top\alpha  \in \partial \psi(\theta(\alpha))$. Hence, 

Hence, 
\begin{align}
{ \rm gap}(\alpha) &=  - \psi^*\left(- \frac{1}{n}X^\top\alpha^\star\right)  +  \psi^*\left(- \frac{1}{n}X^\top\alpha\right) - \frac{1}{n}\sum_{i =1}^n \phi_i^* (\alpha_i^\star) + \frac{1}{n}\sum_{i =1}^n \phi_i^* (\alpha_i) \notag  \\
&= - \left( -\left\langle \frac{1}{n}X^\top\alpha^\star, \theta^\star \right\rangle - \psi(\theta^\star) \right) +  \left( -\left\langle \frac{1}{n} X^\top\alpha, \theta(\alpha) \right\rangle - \psi(\theta(\alpha)) \right)   - \frac{1}{n}\sum_{i =1}^n \phi_i^* (\alpha_i^\star) + \frac{1}{n}\sum_{i =1}^n \phi_i^* (\alpha_i) \notag  \\
&= \psi(\theta^\star)  - \psi(\theta(\alpha))  + \left\langle \frac{1}{n}X^\top\alpha^\star, \theta^\star \right\rangle - \left\langle \frac{1}{n} X^\top\alpha, \theta(\alpha) \right\rangle  - \frac{1}{n}\sum_{i =1}^n \phi_i^* (\alpha_i^\star) + \frac{1}{n} \sum_{i =1}^n \phi_i^* (\alpha_i)\notag  \\
&= \psi(\theta^\star)  - \psi(\theta(\alpha))  + \left\langle \frac{1}{n} X^\top\alpha^\star, \theta^\star \right\rangle - \left\langle  \frac{1}{n} X^\top \alpha, \theta(\alpha) \right\rangle  - \frac{1}{n}\sum_{i =1}^n \phi_i^* (\alpha_i^\star) + \frac{1}{n}\sum_{i =1}^n \phi_i^* (\alpha_i) \notag  \\
&= \psi(\theta^\star)  - \psi(\theta(\alpha))  + \left\langle \frac{1}{n}X^\top\alpha^\star, \theta^\star \right\rangle+\left\langle  \frac{1}{n}X^\top\alpha, \theta^\star \right\rangle - \left\langle  \frac{1}{n}X^\top\alpha, \theta^\star \right\rangle- \left\langle \frac{1}{n} X^\top\alpha, \theta(\alpha) \right\rangle \notag \\
&\qquad \qquad \qquad \qquad \qquad \qquad \qquad \qquad  \qquad  \qquad   - \frac{1}{n}\sum_{i =1}^n \phi_i^* (\alpha_i^\star) + \frac{1}{n}\sum_{i =1}^n \phi_i^* (\alpha_i) \notag  \\
&= \psi(\theta^\star)  - \psi(\theta(\alpha))  - \left\langle  \frac{1}{n}X^\top\alpha, \theta(\alpha) - \theta^\star\right\rangle + \left\langle \frac{1}{n} X^\top\alpha^\star - \frac{1}{n} X^\top\alpha, \theta^\star \right\rangle -\frac{1}{n} \sum_{i =1}^n \phi_i^* (\alpha_i^\star) + \frac{1}{n}\sum_{i =1}^n \phi_i^* (\alpha_i) \notag \\
&= \underbrace{\psi(\theta^\star)  - \psi(\theta(\alpha))  - \langle \nabla \psi(\theta(\alpha)), \theta^\star - \theta(\alpha) \rangle }_{:=D_{\Psi}(\theta^\star , \theta(\alpha))}+ \left\langle \frac{1}{n} X^\top\alpha^\star - \frac{1}{n} X^\top \alpha, \theta^\star \right\rangle \notag \\
& \qquad \qquad \qquad \qquad \qquad \qquad \qquad \qquad \qquad \qquad \qquad  - \frac{1}{n}\sum_{i =1}^n \phi_i^* (\alpha_i^\star) +\frac{1}{n} \sum_{i =1}^n \phi_i^* (\alpha_i) \notag  \\
&= D_{\Psi}(\theta^\star , \theta(\alpha)) + \left\langle \frac{1}{n} \alpha^\star -  \frac{1}{n} \alpha, X \theta^\star \right\rangle - \frac{1}{n}\sum_{i =1}^n \phi_i^* (\alpha_i^\star) + \frac{1}{n}\sum_{i =1}^n \phi_i^* (\alpha_i) \notag  \\
&= D_{\Psi}(\theta^\star , \theta(\alpha)) + \frac{1}{n} \sum_{i=1}^n(\alpha_i^\star - \alpha_i) \cdot x_i^\top\theta^\star   - \frac{1}{n}\sum_{i =1}^n \phi_i^* (\alpha_i^\star) + \frac{1}{n} \sum_{i =1}^n \phi_i^* (\alpha_i) \notag  \\
&= D_{\Psi}(\theta^\star , \theta(\alpha)) - \frac{1}{n} \sum_{i=1}^n(\alpha_i- \alpha_i^\star ) \cdot \nabla \phi^*(\alpha_i^\star)   - \frac{1}{n}\sum_{i =1}^n \phi_i^* (\alpha_i^\star) + \frac{1}{n}\sum_{i =1}^n \phi_i^* (\alpha_i)\notag   \\
& =D_{\Psi}(\theta^\star , \theta(\alpha))  + \frac{1}{n} \sum_{i=1}^n D_{\phi_i^*}( \alpha_i , \alpha_i^\star)  \geq D_{\Psi}(\theta^\star , \theta(\alpha)) .
\end{align}

\end{proof}
After we provide the general result in Proposition~\ref{prop:dual_sub_opt_divergence_gen}, we now provide the proof for proposition~\ref{prop:gen_result_primal_dual} below.  The result in statement is a useful result and can be useful in several ways. For example, the guarantees for SDCA~\citep{shalev2013stochastic,shalev2014accelerated}. We provide the details in the Appendix~\ref{app:SDCA}.

\begin{proof}[Proof of Proposition~\ref{prop:gen_result_primal_dual}]
We can just use the result in Proposition~\ref{prop:dual_sub_opt_divergence_gen} to prove Proposition~\ref{prop:gen_result_primal_dual}.  Let's recall once again the primal dual formulation of the problem which we have in Equation~\eqref{eq:primal_eq} and Equation~\eqref{eq:dual_eq}.
\begin{align}
  &\min_{\theta \in \rb^d} D_\psi(\theta, \theta^{(0)}) \mbox{ such that }  \forall i \in \{1,\dots,n\}, \ x_i^\top \theta \in \mathcal{Y}_i \label{eq:primal_eq_app} \\
  =&\min_{\theta \in \rb^d}  \psi(\theta) + \frac{1}{n} \sum_{i=1}^n \max_{\alpha_i \in \rb^k} \Big\{  \alpha_i^\top x_i^\top \theta - \sigma_{\mathcal{Y}_i}(\alpha_i) \Big\} \notag
\\
 =&\max_{\forall i, \ \alpha_i \in \rb^k} - \frac{1}{n} \sum_{i=1}^n  \sigma_{\mathcal{Y}_i}(\alpha_i)  - \psi^\star\Big(
-\frac{1}{n} \sum_{i=1}^n x_i \alpha_i \label{eq:dual_eq_app} 
\Big) \\
=&\max_{\alpha \in \rb^{n \times k} } G(\alpha) \notag,
 \end{align}
Let $\mathcal{K}_i$ represents that set for all $\theta$ such that $x_i^\top \theta \in \mathcal{Y}_i$ and the indicator function $\iota_{\mathcal{K}_i}$ for a convex set $\mathcal{K}_i$ for all $\in \{1,\dots,n\}$ is defined as $\iota_{\mathcal{K}_i}(x_i^\top \theta) = 0 $ if $x_i^\top \theta \in \mathcal{Y}_i$ and $\iota_{\mathcal{K}_i}(x_i^\top \theta) = +\infty$, otherwise for all $\in \{1,\dots,n\}$. We can write Equation~\eqref{eq:primal_eq_app} in the form of generalized equation given in Equation~\eqref{eq:primal_gen_erm} considering $\phi_{i}(x_i^\top \theta) = \iota_{\mathcal{K}_i}(x_i^\top \theta)$. It is easy to see that $\phi_i^*({\alpha_i}) = \sigma_{\mathcal{Y}_i}(\alpha_i)$. Hence, now the statement follows from Proposition~\ref{prop:dual_sub_opt_divergence_gen}.
\end{proof}

\subsection{Coordinate Descent Update: Proof of Lemma~\ref{lem:RCD_update}} \label{app:coord_descent_update}
We have:

\begin{align}
  \alpha^{(t)}_{i(t)} & =   \arg\max_{\alpha_{i(t)}} - \frac{1}{n} \sigma_{\mathcal{Y}_{i(t)}}(\alpha_i) + \frac{1}{n}  \nabla \psi^\star\Big(
- \frac{1}{n} \sum_{i=1}^n x_i \alpha_i^{(t-1)}
\Big)^\top x_{i(t)}  [ \alpha_{i(t)}  - \alpha^{(t-1)}_{i(t)} ] -  \frac{L_{i(t)}}{2n^2} \|  \alpha_i - \alpha^{(t-1)}_{i(t)}\|_2^2
\notag \\
& =   \arg\max_{\alpha_{i(t)}}  -\frac{1}{n}  \sigma_{\mathcal{Y}_{i(t)}}(\alpha_i) + \frac{1}{n}  \theta(\alpha^{(t-1)})^\top x_{i(t)}  [ \alpha_{i(t)}  - \alpha^{(t-1)}_{i(t)} ] -  \frac{L_{i(t)}}{2n^2} \|  \alpha_i - \alpha^{(t-1)}_{i(t)}\|_2^2  \notag\\
& =   \arg\min_{\alpha_{i(t)}}  \sigma_{\mathcal{Y}_{i(t)}}(\alpha_i)  +  \frac{L_{i(t)}}{2n} \|  \alpha_i - \alpha^{(t-1)}_{i(t)} - \frac{n}{L_{i(t)}} x_{i(t)} ^\top \theta(\alpha^{(t-1)}) \|_2^2. \label{eq:update_rcd_general}
\end{align}

The minimization problem in Equation~\eqref{eq:update_rcd_general} can be written as follows:
\begin{align}
    \begin{split}
      &\min_{\alpha_{i(t)}}  \left[\sigma_{\mathcal{Y}_{i(t)}}(\alpha_i)  +  \frac{L_{i(t)}}{2n} \|  \alpha_i - \alpha^{(t-1)}_{i(t)} - \frac{n}{L_{i(t)}} x_{i(t)} ^\top \theta(\alpha^{(t-1)}) \|_2^2\right] \\
       = &\min_{\alpha_{i(t)}}  \left[\sigma_{\mathcal{Y}_{i(t)}}(\alpha_i)  - \sup_{z}\left[ (\alpha_i - \alpha^{(t-1)}_{i(t)} - \frac{n}{L_{i(t)}} x_{i(t)} ^\top \theta(\alpha^{(t-1)}))^\top z + \frac{n}{2L_{i(t)}}\| z\|^2 \right] \right] \\
       = & \sup_{z\in \mathcal{Y}_{i(t)}}  \left[ - \frac{n}{2L_{i(t)}}\| z\|^2 + z^\top\left (\frac{n}{L_{i(t)}} x_{i(t)} ^\top \theta(\alpha^{(t-1)}) + \alpha^{(t-1)}_{i(t)} \right) \right]\\
    \end{split}
\end{align}
The above maximization problem has a solution at $z^\star = \Pi_{\mathcal{Y}_{i(t)}}\left(x_{i(t)} ^\top \theta(\alpha^{(t-1)}) + \frac{L_{i(t)}}{n}\alpha_{i(t)}^{(t-1)} \right)$. However, $z^\star$ is also the solution of the following optimization formulation:
\begin{align*}
    z^\star = \arg\max_{z}\left[ (\alpha_i - \alpha^{(t-1)}_{i(t)} - \frac{n}{L_{i(t)}} x_{i(t)} ^\top \theta(\alpha^{(t-1)}))^\top z + \frac{n}{2L_{i(t)}}\| z\|^2 \right]
\end{align*}
 
 Comparing both the value of $z^\star$, we get the following  update in $\alpha_{i(t)}$ in alternative form
 $$
 \alpha_{i(t)} = \alpha^{(t-1)}_{i(t)} +  \frac{n}{L_{i(t)}} x_{i(t)} ^\top \theta(\alpha^{(t-1)})
 -  \frac{n}{L_{i(t)}}\Pi_{\Y_i} \Big( \frac{L_{i(t)}}{n}  \alpha^{(t-1)}_{i(t)} +  x_{i(t)} ^\top \theta(\alpha^{(t-1)}) \Big),
 $$
 where $\Pi_{\Y_i}$ is the orthogonal projection on $\Y_i$.

 \subsection{Mirror Descent: [Proof of Theorem~\ref{thm:sharan_mirror_desc}]} \label{app:mirror_descent_sharan}

 The convergence rate does depend on $\psi(\theta^\star)$ but this is not an explicit regularization. The proof goes as follows:

Mirror descent with the mirror map $\psi $ selects $i(t)$ at random and the iteration is
$$
\psi'(\theta^{(t)}) = \psi'(\theta^{(t-1)}) - \gamma x_{i(t)} ( \Pi_{\mathcal{Y}_i}(  x_{i(t)}^\top \theta^{(t-1)} ) - x_{i(t)}^\top \theta^{(t-1)} ) .
$$
Following the proof of \citet{flammarion2017stochastic}, we have for any $\theta \in \rb^d$:
\BEAS
D_\psi(\theta,\theta^{(t)})
& = & D_\psi(\theta,\theta^{(t)}) - D_\psi(\theta^{(t)},\theta^{(t-1)}) + \gamma f'_t(\theta^{(t-1)})^\top ( \theta^{(t)} - \theta) \\
& \leqslant & D_\psi(\theta,\theta^{(t)}) - \frac{\mu}{2} \| \theta^{(t)}- \theta^{(t-1)}\|^2 + \gamma f'_t(\theta^{(t-1)})^\top ( \theta^{(t-1)} - \theta) \\ & & + \gamma \| f'_t(\theta^{(t-1)})\|_\star \| \theta^{(t-1)} - \theta^{(t)}\| \\
& \leqslant & D_\psi(\theta,\theta^{(t)}) -    \gamma f'_t(\theta^{(t-1)})^\top ( \theta^{(t-1)} - \theta) + \frac{\gamma^2}{2\mu} \| f'_t(\theta^{(t-1)})\|_\star^2.
\EEAS
For $\theta = \theta^\star$ and using 
$\E \big[ \| f'_t(\theta^{(t-1)})\|_\star^2 \big] \leqslant \sup_{i}  \| x_i\|_{ 2 \to \star }^2  \big[ f(\theta) - f(\theta^\star) \big]$, we get
and taking expectations, we get:
\BEAS
\big( 1 - \gamma \frac{   \| x_i\|_{ 2 \to \star }^2 }{2\mu} \big)\E \big[ f(\theta^{(t-1)}) - f(\theta^\star) \big] \leqslant \frac{1}{\gamma}\Big( \E[ D_\psi(\theta^\star,\theta^{(t)})  ] - \E\big[ D_\psi(\theta^\star,\theta^{(t-1)})  \big] \Big).
\EEAS
Thus, with $\gamma = \mu / \sup_{i} \| x_i\|_{ 2 \to \star }^2$, we get
$$
\E \big[ f(\theta^{(t-1)}) - f(\theta^\star) \big]\leqslant \frac{2}{\gamma}\Big( \E[ D_\psi(\theta^\star,\theta^{(t)})  ] - \E\big[ D_\psi(\theta^\star,\theta^{(t-1)})  \big] \Big).
$$  
 This leads to
 $$
\E \big[  f(\bar{\theta}_t) - f(\theta^\star) \big]
\leqslant \frac{2}{\gamma t} D_\psi(\theta^\star,\theta^{(0)}) .
 $$

\section{$\ell_p$-perceptron} 
\label{app:l_p_percept}

We start with the proof of Lemma ~\ref{lem:mistake_l_p}.

\begin{proof}
For all $i$, $x_i^\top \theta^\star\geq 1$. Hence,
\begin{align*}
    x_i^\top \theta_t &= x_i^\top \theta_t -  x_i^\top \theta^\star + x_i^\top \theta^\star   = x_i^\top \theta^\star  - x_i^\top (\theta^\star  - \theta_t) \\
    &\geq 1 - x_i^\top (\theta^\star  - \theta_t) \geq 1 - \| x_i\|_q \|\theta_t - \theta^\star\|_p \\
    &\geq 1 - R \|\theta_t - \theta^\star\|_p.
\end{align*}
Assuming $\alpha_0  = 0$, from Equation \eqref{eq:acc_percet_p}, we have
\begin{align*}
     \E\Big[\| \theta_t- \theta^\star\|_p\Big] \leq \frac{    2\sqrt{2}      \max_i \|x_i \|_q}{\sqrt{(p-1)}t} \sqrt{\frac{G(\alpha^\star) - G(0)}{ \max_i \|x_i \|_q} + \frac{1}{2}\| \alpha^\star\|^2}
\end{align*}
Now for on average for no mis-classification for all $i \in \{ 1, \cdots, n\}$,
\begin{align}
    1 \geq R \E\Big[\|\theta_t - \theta^\star\|_p\Big] \Rightarrow t \geq \frac{2\sqrt{2}R^2 }{\sqrt{p-1}  } \sqrt{\frac{G(\alpha^\star) - G(0)}{ R} + \frac{1}{2}\| \alpha^\star\|^2}.
\end{align}

\end{proof}

\paragraph{Mistake Bound $\ell_p$-primal perceptron.} If we apply mirror descent with the mirror map $\psi = \frac{1}{2} \| \cdot \|_p^2$ to the minimization of $\frac{1}{n} \sum_{i=1}^n ( 1 - \theta^\top x_i)_+$, then the iteration is
$$
\psi'(\theta_{t}) = \psi'(\theta_{t-1}) - \gamma 1_{ 1 - \theta_{t-1}^\top x_{i(t)} > 0} x_{i(t)},
$$
and we have
$$
\frac{1}{n} \sum_{i=1}^n ( 1 - \bar{\theta}_t^\top x_i)_+ \leqslant  \frac{\| \theta_\star\|_p^2}{2\gamma t }
+ \gamma \frac{\max_{i} \| x_i \|_q^2 }{2(p-1)}.
$$
The best $\gamma$ is equal to 
$\gamma = \frac{ \| \theta_\star\|_p }{ \max_{i} \| x_i\|_q} \frac{\sqrt{p-1}}{\sqrt{t}}$, which does depend on too many things, and leads to a proportion of mistakes on the training set less than
$$
 \frac{ \| \theta_\star\|_p  \max_{i} \| x_i\|_q  } {\sqrt{p-1} \sqrt{t}}.
$$

\subsection{Update for Random Coordinate Descent}

We have:
\BEAS
 & & \min_{ \theta \in \rb^d} \frac{1}{2} \| \theta \|_p^2 \mbox{ such that } X\theta \geqslant 1 \\
 & = &  \min_{ \theta \in \rb^d} \max_{\alpha \in \rb^n} \frac{1}{2} \| \theta \|_p^2 + \alpha^\top ( 1 - X\theta) \\
 & = &  \max_{\alpha \in \rb^n}  - \frac{1}{2} \| X^\top \alpha \|_q^2 + \alpha^\top  1  ,
 \EEAS
 where, at optimality, $\theta $  can be obtained from $X^\top \alpha$ as 
 $$
 \theta_j = \| X^\top \alpha \|_q^{2-q}  (X^\top \alpha)_j^{q-1}
  ,$$
   where we define $u^{q-1} = |u|^{q-1}  {\rm sign}(u)$.
   
    The function $ \frac{1}{2} \| X^\top \alpha \|_p^2$ is smooth, and the regular smoothness constant with respect to the $i$-th variable which is less than 
 $$
 L_i = \frac{1}{p-1} \| x_i \|_q^2.
 $$
 A dual coordinate ascent step corresponds to choosing $i(t)$ and replacing $(\alpha_{t-1})_{i(t)}$ by 
 $$(\alpha_t)_i = \max \Big\{ 0, (\alpha_{t-1})_{i(t)} + \frac{1}{L_{i(t)}} 
 \big( 1  - \| X^\top \alpha_{t-1} \|_q^{2-q} \sum_{j=1}^d  [ (X^\top \alpha_{t-1})_j ]^{q-1} X_{i(t)j}
 \Big\},$$
 which can be interpreted as:
 $$
 (\alpha_t)_i = \max \Big\{ 0, (\alpha_{t-1})_{i(t)} + \frac{1}{L_{i(t)}} 
 \big( 1- \theta_{t-1}^\top x_{i(t)} \big)
 \Big\}.
 $$
 
 \subsection{$\ell_2$-perceptron} \label{app:l_2_percept}

 The primal problem has the following dual form under the interpolation regime
\begin{align*}
\max_{\alpha \geq 0, \alpha \in \mathbb{R}^n } ~ \alpha^\top {1} - \frac{1}{2} \| X \alpha \|^2 .
\end{align*}
We denote $S_v$ as the set of support vectors \textit{i.e.} $S_v$ is the set of indices where $\alpha_j^\star \neq 0$. Hence, we also have $\tilde x_j ^\top \theta^\star = 1$ for $j \in S_v $. $\alpha_{S_v}$ denotes the vector of non-zero entries in $\alpha$. Correspondingly, ${X}_{S_v}$ denotes the feature matrix for support vectors. From the first order suboptimality condition we have, 
\begin{align*}
\theta(\alpha)  = \frac{1}{n}X \alpha.
\end{align*}
We also know that for support vectors, $y_i\cdot x_i^\top \theta^\star = \tilde x_i^\top \theta^\star = 1$ for all $i \in S_v$. Also $\theta^\star = \frac{1}{n} X_{S_v} \alpha^\star_{S_v}$. Hence, 
\begin{align*}
\frac{1}{n}X_{S_v}^\top X_{S_v} \alpha^\star_{S_v} = {1} \Rightarrow \alpha^\star_{S_v}  = n (X_{S_v}^\top X_{S_v})^{-1} {1}.
\end{align*}
From Lemma~\ref{lem:mistake_l_p}, we should have  $t \geq  ~ \frac{2\sqrt{2}R^2 }{\sqrt{p-1}  } \sqrt{\frac{G(\alpha^\star) - G(0)}{ R} + \frac{1}{2}\| \alpha^\star\|^2} $, for no training mistakes. 

We now use Corollary~\ref{cor:percept_sharna} to get mistake bound on the perceptron. To have no mistakes on average, the proportion of mistakes should be less than $1/n$. Hence, 
\begin{align}
\frac{R \|\theta^\star \|}{\sqrt{t}} \leq \frac{1}{n} \Rightarrow t \geq R^2 \| \theta^\star\|^2 n^2.   \label{eq:sgd_mistake_cond}
\end{align}
We already have $\alpha^\star_{S_v}  = n (X_{S_v}^\top X_{S_v})^{-1} {1}$. 
\begin{align*}
\theta^\star = \frac{1}{n}X \alpha^\star = \frac{1}{n} X_{S_v} \alpha^\star_{S_v} = X_{S_v}  (X_{S_v}^\top X_{S_v})^{-1} {1}.
\end{align*}
Finally we have the following:
\begin{align}
\begin{split}
\|\alpha^\star \| &= \| \alpha_{S_v}^\star\| = n \|(X_{S_v}^\top X_{S_v})^{-1} {1}\| \\
\|\theta^\star\|^2 &= \| X_{S_v}  (X_{S_v}^\top X_{S_v})^{-1} {1}\|^2 = {1}^\top  (X_{S_v}^\top X_{S_v})^{-1} {1}.
\end{split}
\end{align}
Hence, one can compare the number of minimum iteration required by both the approaches.

\section{(Accelerated) Stochastic Dual Coordinate Descent} \label{app:SDCA}

Stochastic dual coordinate ascent \citep{shalev2013stochastic} is a popular approach to optimize regularized empirical risk minimize problem. For this section, let $\phi_1, \cdots,\phi_n$ be a sequence of $\frac{1}{\gamma}$-smooth convex losses and let $\lambda > 0$ be a regularization parameter then consider following regularized empirical risk minimization problem:
\begin{align}
\min_{\theta \in \mathbb{R^d} } \Bigg[ \mathcal{S}_{P}(\theta):= \frac{\lambda}{2} \| \theta\|^2  + \frac{1}{n}\sum_{i=1}^{n} \phi_i(X_i^\top\theta)   \Bigg]. \label{eq:primal_gen_erm_sdca}
\end{align}
Corresponding dual problem of the minimization problem given in equation~\eqref{eq:primal_gen_erm_sdca} can be written similarly as:
\begin{align}
\max_{\alpha \in \mathbb{R}^n}~& \Bigg[ \mathcal{S}_{D}(\alpha):= - \frac{\lambda}{2} \| \frac{1}{\lambda n} X^\top\alpha \|^2 - \frac{1}{n}\sum_{i=1}^n \phi_i^* (-\alpha_i) \Bigg] \label{eq:dual_gen_erm_sdca}
\end{align}
There is one to one relation between the smoothness constant and strong convexity parameter of primal and corresponding dual function. We prove the following result from~\citet{kakade2009duality}.
\begin{theorem}[Theorem 6, \citep{kakade2009duality}] 
Assume that $f$ is a closed and convex function. Then $f$ is $\beta$-strongly convex \textit{w.r.t.} a norm $\|\cdot \|$ if  and only if $f^*$ is $\frac{1}{\beta}$-smooth \textit{w.r.t.} the dual norm $\|\cdot \|_{*}$.
\end{theorem}
From the above theorem it is clear that $\phi_i^*$ are $\gamma$-strongly convex. Hence the term $\frac{1}{n}\sum_{i=1}^n \phi_i^* (-\alpha_i)$ is $\frac{\gamma}{n}$ strongly convex.  Similary coordinate wise smoothness $L_i  = \frac{\|x_i \|^2}{ \lambda n^2}$.\\ 
Now, just as a direct implication of the result provided in Proposition~\ref{prop:dual_sub_opt_divergence_gen}, we have the convergence result for SDCA~\citep{shalev2013stochastic} and accelerated stochastic dual coordinate ascent~\citep{shalev2014accelerated} which we provide in Corollary~\ref{cor:result_SDCA} and Corollary~\ref{cor:result_Acc-SDCA}. For the next two results, we denote $\theta_k$ as $\theta(\alpha_k)$. 
\begin{cor}[Stochastic Dual Coordinate Ascent] \label{cor:result_SDCA}
Consider the regularized empirical risk minimization problem given in equation~\eqref{eq:primal_gen_erm_sdca}, then if we run SDCA~\citep{shalev2013stochastic} algorithm starting from $\alpha_0 \in \mathbb{R}^n$ with a fix step size $1/\max_i L_i$ where $L_i = \frac{\|x_i \|^2}{ \lambda n^2}$,  primal iterate after $k$ iterations converges as following:
\begin{align*}
\frac{\lambda}{2} \|\theta_{k+1} - \theta^\star \|^2 \leq D(\alpha_{k+1})  \leq  \left( 1 - \frac{\gamma \lambda}{ \max_i \|x_i \|^2}\right)^k (\mathcal{S}_{D}(\alpha_0) - \mathcal{S}_{D}(\alpha^\star)).
\end{align*}
\end{cor}
\begin{proof}
From~\citet{allen2016even}, it is clear that for $\mu$-strongly convex and  $L_i$-coordinate wise smooth convex function $\mathcal{S}_{D}(\alpha)$ where $\alpha \in \mathbb{R}^n$, randomized coordinate descent has the following convergence  guarantee:
\begin{align*}
D(\alpha_{k+1}) \leq  \left( 1 - \frac{\mu}{n \max_i L_i}\right)^k (\mathcal{S}_{D}(\alpha_0) - \mathcal{S}_{D}(\alpha^\star)).
\end{align*}
Here, $\mu = \frac{\gamma}{n}$. First part of the inequality directly comes from Proposition~\ref{prop:dual_sub_opt_divergence_gen} by the observation that here $\psi(\cdot)=\frac{\lambda}{2} \| \cdot\|^2$ and bregman divergence are always positive. 
\end{proof}
\begin{cor}[Accelerated Stochastic Dual Coordinate Ascent] \label{cor:result_Acc-SDCA}
Consider the regularized empirical risk minimization problem given in equation~\eqref{eq:primal_gen_erm_sdca}, then if we run Accelerated SDCA~\citep{shalev2014accelerated} algorithm starting from $\alpha_0 \in \mathbb{R}^n$, we have following convergence rate for the primal iterates:
\begin{align*}
\frac{\lambda}{2} \|\theta_{k+1} - \theta^\star \|^2 \leq D(\alpha_{k+1})  & \leq 2 \left( 1 - \frac{\sqrt{\gamma \lambda}}{\sqrt{ \max_i \| x_i\|^2}}\right)^k (\mathcal{S}_{D}(\alpha_0) - \mathcal{S}_{D}(\alpha^\star)).
\end{align*}
\end{cor}
\begin{proof}
From~\citet{allen2016even}, it is clear that for $\mu$-strongly convex and $L_i$-coordinate wise smooth convex function $\mathcal{S}_{D}(\alpha)$ where $\alpha \in \mathbb{R}^n$, accelerated randomized coordinate descent has the following convergence guarantee:
\begin{align*}
D(\alpha_{k+1}) \leq  2 \left( 1 - \frac{\sqrt{\mu}}{n\sqrt{ \max_i L_i}}\right)^k (\mathcal{S}_{D}(\alpha_0) - \mathcal{S}_{D}(\alpha^\star)).
\end{align*}
First part of the inequality directly comes from Proposition~\ref{prop:dual_sub_opt_divergence_gen} by the observation that here $\psi(\cdot)=\frac{\lambda}{2} \| \cdot\|^2$ and bregman divergence are always positive. Here  $\mu = \frac{\gamma}{n}$ and $L_i = \frac{\|x_i \|^2}{ \lambda n^2}$.
\end{proof}

\paragraph{Discussion.} Let us denote duality gap at dual variable $\alpha$ as $\Delta(\alpha)$. From the definition of the duality gap $\Delta(\alpha)  =  \mathcal{S}_{P}(\theta(\alpha)) -  \mathcal{S}_{D}(\alpha)$. However, $\Delta(\alpha)$ is an upper bound on the primal sub-optimality gap as well on dual  sub-optimality gap. The main difference in the analysis presented in our work with the works of \citet{shalev2013stochastic} and \citet{shalev2014accelerated} is that the we provide the guarantee in term of the iterate. However \citet{shalev2013stochastic} and  \citet{shalev2014accelerated} provide convergence in terms of duality gap $\Delta(\alpha)$. Another main difference is that we use constant step size in each step and the output of our algorithm doesn't need averaging of the past iterates. Our analysis holds for the last iterate.

 \clearpage

\end{document}